\documentclass[12pt,reqno]{amsart}
\usepackage[all]{xy}
\usepackage{graphicx}
\usepackage[utf8]{inputenc}
\usepackage{amssymb}
\usepackage{hyperref}
\usepackage[utf8]{inputenc}
\usepackage{yfonts}

\newcommand{\C}{{\mathbb C} }

\newcommand{\cA}{{\mathcal A} }
\newcommand{\cB}{{\mathcal B} }

\newcommand{\cE}{{\mathcal E} }

\newcommand{\cM}{{\mathcal M} }
\newcommand{\cN}{{\mathcal N} }
\newcommand{\cO}{{\mathcal O} }

\newcommand{\cX}{{\mathcal X} }

\newcommand{\wt}{\widetilde}

\newtheorem{theorem}{Theorem}[section]

\newtheorem{proposition}[theorem]{Proposition}

\theoremstyle{definition}
\newtheorem{definition}[theorem]{Definition}

\newtheorem*{fact*}{Fact}

\numberwithin{equation}{section}

\def\ul#1{\underline{#1}}
\def\ks{Ko\-dai\-ra-Spen\-cer }
\def\C{\mathbb{C}}

\def\isom{\stackrel{\sim}{\longrightarrow}}

\def\cage{Car\-tan geometry}
\def\cages{Car\-tan geometries}

\def\fg{\mathfrak g}
\def\fh{\mathfrak h}

\title[Deformation theory of holomorphic Cartan geometries]{Deformation theory of
holomorphic Cartan geometries}

\author[I.~Biswas]{Indranil Biswas}
\address{School of Mathematics, Tata Institute of Fundamental Research, Homi Bhabha Road, Mumbai 400005, India}
\email{indranil@math.tifr.res.in}

\author[S.~Dumitrescu]{Sorin Dumitrescu}
\address{Universit\'e C\^ote d'Azur, CNRS, LJAD, France}
\email{dumitres@unice.fr}

\author[G.~Schumacher]{Georg Schumacher}
\address{Fachbereich Mathematik und Informatik,
Philipps-Universit\"at Marburg, Lahnberge, Hans-Meerwein-Strasse, D-35032
Marburg, Germany}
\email{schumac@mathematik.uni-marburg.de}



\begin{document}

\begin{abstract}
We introduce the deformation theory of holomorphic Cartan geometries. The infinitesimal
automorphisms, as well as the infinitesimal deformations, of
holomorphic Cartan geometries are computed. We also prove the existence of a semi-universal
deformation of a holomorphic Cartan geometry.
\end{abstract}

\maketitle

\section{Introduction}

Let $G$ be a complex Lie group and $H \,<\, G$ a complex Lie subgroup with Lie algebras $\fg$, and $\fh$ respectively. The quotient
map $G\,\longrightarrow\, G/H$ defines a holomorphic principal $H$-bundle. Moreover, on the total space of this principal
bundle, namely $G$, we have a tautological $\fg$--valued holomorphic $1$-form (the Maurer-Cartan form), which is constructed by
identifying $\fg$ with the right-invariant vector fields. This $1$-form is an isomorphism of the holomorphic tangent bundle
of $G$ with the trivial vector bundle $G\times\fg\, \longrightarrow\, G$.
The restriction of this form to the fibers of the projection $G\,\longrightarrow\, G/H$ coincide with the
$\fh$-valued Maurer-Cartan form for the right-action of $H$ on the fiber.

A holomorphic \cage\ of type $(G,\, H)$ on a
compact complex manifold $X$ is infinitesimally modeled on the above set-up. More precisely,
a holomorphic \cage\ of type $(G,\, H)$ on $X$ consists of
\begin{itemize}
\item a holomorphic principal
$H$-bundle $E_H$ over $X$,

\item a holomorphic $1$-form $A$ on $E_H$ with values in the Lie algebra $\fg$
of $G$ that induces a holomorphic isomorphism from the holomorphic tangent bundle of $E_H$ to the trivial vector bundle
on $E_H$ with fiber $\fg$. This isomorphism is required to be $H$-invariant and on each fiber
of $E_H$ it should be the Maurer-Cartan form for the action of $H$.
\end{itemize}
(see Definition \ref{de:Cartangeom} in Section \ref{s2} and for more details \cite{Sh}). The notion
of holomorphic \cage\ extends to Sasakian manifolds \cite{BDS}.

A fundamental result of E. Cartan shows that the obstruction for $A$ to satisfy the Maurer-Cartan equation of $G$ is a curvature tensor which 
vanishes if and only if $(E_H,\, A)$ is locally isomorphic (not just infinitesimally) to the $H$-principal bundle $G \, \longrightarrow\, G/H$ 
endowed with the Maurer-Cartan form \cite{Sh}. In this case $X$ admits local coordinates with values in $G/H$ which are well defined up to the 
canonical action of $G$ on $G/H$. Indeed, Ehresmann proved that a flat Cartan geometry is equivalent with the following data: a holomorphic 
principal $G$-bundle over $X$ endowed with a flat holomorphic connection and a $H$-subbundle transverse to the flat connection \cite{Eh} (see also 
the survey \cite{BD1}). This implies that the pull-back of this $G$-principal bundle to the universal cover $\tilde{X}$ of $X$ is isomorphic to 
$\tilde X \times G$ and the $H$-subbundle is given by a holomorphic map $\widetilde{X} \,\longrightarrow\, G/H$. Moreover, the transversality 
condition is equivalent to 
the fact that the previous map is a local biholomorphism, traditionally called the developing map of the (flat) Cartan geometry. The developing map 
is equivariant with respect to the action of the fundamental group $\pi_1(X)$ of $X$ by deck transformations on $\tilde{X}$ and through the 
monodromy morphism $\rho \,:\, \pi_1(X) \,\longrightarrow\, G$ (of the $G$-flat bundle) on the model space $G/H$ \cite{Eh} (see also \cite{BD1}). 
Ehresmann's 
geometrical description leads to the following nice and useful description of the deformation space of flat Cartan geometries with given model type 
$(G,\,H)$ on the (real) manifold $X$. This so-called Ehresmann-Thurston principle ensures that the map associating to each flat Cartan geometry its 
monodromy morphism $\rho \,:\, \pi_1(X)\, \longrightarrow\, G$ (uniquely defined up to inner conjugacy in $G$) is a local homeomorphism between the 
deformation space of flat Cartan geometries with model $(G,\,H)$ on $X$ and the space of group homomorphisms from $\pi_1(X)$ to $G$ \cite{Go}.

Since $G$ is a complex Lie group and $H$ a closed complex subgroup in $G$, the model manifold $G/H$ inherits a 
$G$-invariant complex structure. Any flat Cartan geometry with model $(G,\,H)$ induces on $X$ an underlying 
complex structure. Hence there is a natural forgetful map from the deformation space of flat Cartan geometries 
with model complex Lie groups $(G,\,H)$ into the Kuranishi space of $X$. In the particular case of complex 
projective structures on Riemann surfaces, this map played a major role in the understanding of the 
uniformization theorem for Riemann surfaces (see \cite{Gu} or \cite{StG}). More precisely, the uniformization 
theorem asserts the existence on any Riemann surface of a compatible complex projective structure with injective 
developing map. In the case where the Riemann surface is compact of genus $g \geq 2$, the developing map is an 
isomorphism between the universal cover of the surface and the unitary disk in $\mathbb C$ and the image of the 
monodromy morphism is a uniform lattice in ${\rm PSL}(2, \mathbb R)$ (the isometry group of the hyperbolic disk) 
\cite{StG}.

More recently the deformation space of flat Cartan geometries with model \linebreak $G\,=\,{\rm SL}(2, 
\mathbb C) \times {\rm SL}(2, \mathbb C)$ and $H\,=\,{\rm SL}(2, \mathbb C)$, diagonally embedded, was used 
by Ghys in \cite{Gh} in order to compute the Kuranishi space of the parallelizable manifolds ${\rm SL}(2, 
\mathbb C) / \Gamma$, with $\Gamma$ a uniform lattice in ${\rm SL}(2, \mathbb C)$. It was proved in \cite{Gh} 
that the deformation space of those flat Cartan geometries is locally isomorphic to the Kuranishi space and 
modeled on the germ at the trivial morphism in the algebraic variety of group homomorphisms (representations) 
from $\Gamma$ into ${\rm SL}(2, \mathbb C)$. In particular, for any uniform lattice $\Gamma$ with positive 
first Betti number this germ has positive dimension. Hence the corresponding parallelizable manifolds ${\rm 
SL}(2, \mathbb C) / \Gamma$ admit nontrivial deformations of the underlying complex structure. Those examples 
of flexible parallelizable manifolds associated to semi-simple complex Lie groups are exotic (the unique case 
not covered by Raghunathan's rigidity results \cite{Ra} is precisely that of a factor locally isomorphic to 
${\rm SL}(2, \mathbb C)$).

Let us mention that the deformation theory of representations of fundamental groups of compact K\"ahler 
manifolds into real algebraic groups $G$ was worked out in \cite{GM}. Under various conditions on a given 
representation, it was proved in \cite{GM} that there exists a neighborhood of it in the algebraic variety of 
representations which is analytically equivalent to a cone defined by homogeneous quadratic equations. This is 
the case in the neighborhood of the monodromy morphism of a variation of Hodge structure. It is also the case 
for the monodromy of a flat principal bundle $G$ which admits a reduction to a compact subgroup $K$ such that 
the quotient $G/K$ is a Hermitian symmetric space.

Our aim here is to introduce the deformation theory of holomorphic (not necessarily flat) \cages\ on a compact 
complex manifold. We compute the tangent cohomology of a holomorphic \cage\ $(E_H,\,A)$ in degree zero and one. 
Infinitesimal automorphisms of a holomorphic \cage\ $(E_H,\,A)$ consist of all the $ad(E_H)$-valued holomorphic 
vector fields $\phi$ satisfying the condition that the Lie derivative $L_\phi(A)$ vanishes. The space of 
infinitesimal deformations of $(E_H,\,A)$ fits into a short exact sequence that we construct in the fourth 
section. The construction of the semi-universal deformation is worked out in the last section. Using the same 
methods a semi-universal deformation can also be constructed when the principal $H$-bundle and the underlying 
compact complex manifold are both moving.

While we have restricted ourselves to the holomorphic category, it should be mentioned that Cartan geometries 
are also defined in $C^\infty$ category. However in the $C^\infty$ category, the infinitesimal deformations of a 
Cartan geometry are not parameterized by a finite dimensional space.

Our project in the future is to use the deformation theory of holomorphic \cages\ on a compact complex manifold developed here in order to 
investigate the Kuranishi space of a compact complex manifold bearing a holomorphic Cartan geometry. Hopefully this will led to a better 
understanding of the classification of compact complex manifolds admitting holomorphic Cartan geometries and to uniformization theorems for those
in higher dimension. It should be mentioned that Inoue, Kobayashi and Ochiai proved that compact complex surfaces bearing holomorphic affine 
connections (respectively, holomorphic projective connections) also admit flat holomorphic affine connections (respectively, flat holomorphic 
projective 
connections) with corresponding injective developing map. In particular, those complex surfaces are uniformized as quotients of open subsets in the 
complex affine plane (respectively, complex projective plane) by a discrete subgroup of affine transformations (respectively, projective 
transformations) 
acting 
properly and discontinuously \cite{IKO, KO1,KO2}. We aim to prove, as a generalization of those results, that compact complex surfaces bearing holomorphic Cartan geometries with model $(G,\,H)$ also admit flat holomorphic
Cartan geometries with model $(G,\,H)$ and with corresponding injective developing map into $G/H$. This would uniformize compact complex surfaces bearing holomorphic Cartan geometries with model $(G,\,H)$ as compact quotients of open subsets $U$ in $G/H$ by discrete subgroups in $G$ preserving $U \subset G/H$ and acting properly and discontinuously on $U$.

\section{Deformations of holomorphic Cartan geometries -- Definition}\label{s2}

For any complex manifold $M$ its holomorphic tangent bundle will be denoted by $T_M$. We
shall denote by $G$ a connected complex Lie group, and by $H\,<\,G$ a closed connected complex Lie subgroup. Furthermore fix a compact complex manifold $X$. Let $E_H$ denote a holomorphic principal $H$-bundle over $X$. Let $E_G\,= \,E_H\times_H G$ be the holomorphic principal $G$-bundle over $X$ obtained by extending the structure group of $E_H$ using the inclusion of $H$ in $G$.

The action of the group $H$ on $E_H$ produces an action of $H$ on the tangent bundle $T_{E_H}$. In other words, for $p \,\in\, E_H$, a tangent vector $v\,\in\, T_p(E_H)$ and $h\,\in\, H$, we have $v\cdot h \,=\, R_{h*}v$, where $R_h$ is the right
multiplication by $h$. For $\gamma\,\in\, \fg\,=\, {\rm Lie}(G)$ we have $h\cdot \gamma\,=\, ad(h^{-1}(\gamma))$.

Let $\pi\,:\,E_H\,\longrightarrow\, X$ be the projection from the total space of the
principal $H$-bundle $E_H$, and let $\pi_* \,:\, T_{E_H} \,\longrightarrow\, T_X$ the corresponding induced map of tangent
bundles. The adjoint bundle $ad(E_H)$ is defined to be the associated vector bundle $E_H\times_H \fh$, where $\fh \,=\, {\rm Lie}(H)$.
The space of vertical tangent vectors $\ker \pi_*$ is invariant under the action of $H$, and we have
\begin{equation}\label{a1}
ad(E_H)\,=\, ker(\pi_*)/H\, .
\end{equation}
Given a section of $ad(E_H)$ we shall use the same notation for its pull-back to a section of $\ker(\pi_*)
\,\subset \,T_{E_H}$.

\begin{definition}\label{de:Cartangeom}
A {\it holomorphic Cartan geometry} of type $(G,\,H)$ on $X$ is a pair $(E_H,\,A)$, where
$E_H$ is a holomorphic principal $H$-bundle on $X$, and $A$ is a
$\C$--linear holomorphic isomorphism of vector bundles
$$
A\,:\,T_{E_H}\,\isom\, E_H \times \fg
$$
\vskip-3mm
over $E_H$ such that
\begin{itemize}
\item[(i)] $A$ is $H$--equivariant, and
\item[(ii)] the restriction of $A$ to the fibers of $E_H\,\longrightarrow\, X$ coincides with the Maurer-Cartan form
on the fiber for the action of $H$.
\end{itemize}

We note that the above isomorphism $A$ induces an isomorphism
\begin{equation}\label{eq:At_ad}
A_H\,:\, T_{E_H}/H \,\isom\, (E_H\times \fg)/H\, ,
\end{equation}
where $T_{E_H}/H$ is, by definition, the Atiyah bundle $At(E_H)$, which fits into the Atiyah exact sequence
$$
0\,\longrightarrow\, ad(E_H) \,\longrightarrow\, At(E_H) \,\longrightarrow\, T_X \,\longrightarrow\, 0
$$
(see \cite{At}). We also have $(E_H\times \fg)/H \,\simeq\, ad(E_G)$, where the quotient is for
the conjugation action mentioned earlier.

The isomorphism $A_H$ in \eqref{eq:At_ad} induces a holomorphic connection on $E_G$
\cite{Sh}, \cite[(2.8)]{BD}, which in turn produces a holomorphic connection on $ad(E_G)$.
Therefore, we have a holomorphic differential operator
\begin{equation}\label{eq:Carconn}
D\,:\, ad(E_G) \,\longrightarrow\, \Omega^1_X \otimes ad(E_G)
\end{equation}
of order one.

An {\it isomorphism} $\Phi\,:\,(E_H,\,A)\,\longrightarrow\, (F_H,\,B)$ between holomorphic Cartan geometries $(E_H,\,A)$ and $ (F_H,\,B)$ of same type $(G, \,H)$ on $X$
is given by a holomorphic isomorphism $\phi\,:\,E_H \,\longrightarrow\, F_H$ of principal $H$-bundles that takes $A$ to
$B$ so that
\begin{equation}\label{eq:iso}
\xymatrix{ &T_{E_H} \ar[r]^A \ar[dl] \ar[dd]_{\phi_*} & E_H \times \fg \ar[dd]^{\phi\times id_\fg}\\X & & \\& T_{F_H} \ar[lu]\ar[r]^B & F_H \times \fg
}
\end{equation}
is a commutative diagram (where $\phi_*$ denotes the differential of $\phi$).
\end{definition}

Our aim is to define deformations of holomorphic Cartan geometries. 

Let us first define families of Cartan geometries of type $(G, \,H)$ on a given complex manifold $X$
and isomorphisms between families.

\begin{definition}\label{de:HolFam}
Let $H<G$ be a closed connected complex subgroup of a complex Lie group, and $S$ a complex space.
\begin{itemize}
\item[(i)]
A {\em holomorphic family} of principal $H$-bundles with
holomorphic Cartan geometries over $S$ (also called {\it holomorphic family of Cartan geometries})
is a pair $(\cE_H,\cA)$, where $\cE_H$ is a principal $H$-bundle over $X\times S$ together with a linear isomorphism $\cA$ over $S$
\begin{equation}\label{eq:famCg}
\xymatrix{T_{\cE_H/S} \ar[r]^\cA_\sim \ar[rd] & \cE_H\times \fg \ar[d]\\ & S}
\end{equation}
such that all restrictions $\cA_s$ of $\cA$ to fibers $T_{\cE_{H,s}}$ over $s\in S$ define holomorphic Cartan geometries $(T_{\cE_{H,s}}, \,\cA_s)$ of type $(G, \,H)$ on $X$.
\item[(ii)]
An {\it isomorphism} $\Xi\,:\, 
(\cE_{H,1},\,\cA_1) 
 \,\longrightarrow\, (\cE_{H,2},\,\cA_2)$ between holomorphic families of Cartan geometries $(\cE_{H,1},\,\cA_1)$ and $(\cE_{H,2},\,\cA_2)$ over $S$ is given by an isomorphism $\xi : \cE_{H,1} \,\isom \,\cE_{H,2}$ of principal $H$-bundles such that $(\xi \times id_\fg)\circ \cA_1\,=\, 
\cA_2\circ T_\xi$, with $T_\xi \,:\, T \cE_{H,1} \,\isom\, T \cE_{H,2}$ being the differential of $\xi$.
\end{itemize}
\end{definition}

Let us now to define deformations (over complex spaces) of a holomorphic Cartan geometry and the corresponding
notion of isomorphism between deformations. We also define below (germs of) deformations (over germs of complex spaces) of a Cartan geometry.

\begin{definition}\label{de:holfam}
Let $(E_H,\, A)$ be a holomorphic Cartan geometry of type $(G, \,H)$ on $X$.

\begin{itemize}
\item[(i)] Let $(S,\,s_0)$ be a complex space with a distinguished point $s_0\,\in\, S$. A {\em deformation} of $(E_H,\, A)$ over $(S,\,s_0)$ is a 
pair $((\cE_H,\, \cA), \,\Phi)$, where
\begin{itemize}
\item[(a)]$(\cE_H,\,\cA)$ is a holomorphic family of holomorphic Cartan geometries over $S$, and

\item[(b)] $\Phi\,:\, (E_H,\,A) \,\isom\, (\cE_{H,s_0}, \,\cA_{s_0})$ is an isomorphism of holomorphic Cartan geometries.
\end{itemize}

An {\it isomorphism} $((\cE_{H1},\, \cA_1),\, \Phi_1) \isom ((\cE_{H1},\, \cA_2),\, \Phi_2)$ between 
deformations $((\cE_{H1},\, \cA_1),\, \Phi_1)$ and $((\cE_{H1},\, \cA_2),\, \Phi_2)$ of the Cartan geometry 
$(E_H,\, A)$ is given by an isomorphism of families $\Xi\,:\, (\cE_{H,1},\,\cA_1) \,\longrightarrow\, 
(\cE_{H,2},\,\cA_2)$ such that $\Xi \circ \Phi_1 \,=\, \Phi_2$.

\item[(ii)] Let $\ul S$ be a germ of a complex space represented by a complex space $(S,s_0)$ with a distinguished point $s_0\in S$. A {\it (germ of)
deformation} $((\ul\cE_H,\,\ul\cA), \,\ul\Phi)$ of $(E_H,\,A)$ over $\ul S$ is an equivalence class of deformations of $(E_H,\,A)$ over complex spaces with 
distinguished point $s_0$. More precisely, a (germ) of deformation over $\ul S$ is represented by a deformation of $(E_H,\,A)$ over a neighborhood $S_1\,\subset\, S$ of $s_0$. A further deformation over 
an open neighborhood $s_0\,\in\, S_2\,\subset\, S$ is equivalent, if there exists a neighborhood $s_0\,\in \,S_3 \,\subset \,S_1\cap S_2$ and an
isomorphism of the restrictions to $S_3$ in the sense of (i).
\end{itemize}
\end{definition}

\section{Deformations of holomorphic Cartan geometries}

The tangent cohomology of a $H$-principal bundle $E_H$ is equal to $H^{\bullet}(X, \,ad(E_H))$, where 
$ad(E_H)$ is the adjoint bundle of $E_H$ with fiber $\fh$ \cite{Don} (see also \cite{BHH}). In particular in 
degrees $0$, $1$, and $2$ we have infinitesimal automorphisms, infinitesimal deformations, and the space 
containing obstructions respectively.

We deal now with the tangent cohomology for holomorphic Cartan geometries.

\subsection{Infinitesimal automorphisms of holomorphic Cartan geometries}

Let $(E_H,\,A)$ be a holomorphic \cage\ of type $(G, \,H)$ on $X$.

We deal first with the tangent cohomology $T^0(E_H,\,A)$ of degree zero for the holomorphic Cartan geometry 
$(E_H, \, A)$. Recall that $A$ takes values in $\mathfrak g$. From \eqref{a1} it follows that any holomorphic 
section $\psi$ of $ad(E_H)$ over $U\, \subset\, X$ gives rise to a $H$--invariant holomorphic vector field 
$\widetilde\psi$ over $E_H\vert_U$ which is vertical for the projection $\pi$. We shall identify $\psi$ with 
$\widetilde\psi$. By $L_\psi$ we denote the Lie derivative with respect to this vector field 
$\widetilde\psi$.

\begin{proposition}\label{pr:infaut}
The space of infinitesimal automorphisms is equal to
\begin{equation}\label{eq:infaut}
T^0(E_H,\,A) \,= \,H^0(X, \,ad(E_H))_A\,= \,\{\psi \,\in\,H^0(X, \,ad(E_H))\,\mid\, L_\psi(A)\,=\,0\}\, .
\end{equation}
\end{proposition}

\begin{proof}
We consider \eqref{eq:iso} for $F_H\,=\,E_H$ and the infinitesimal action of an element $\psi\in H^0(X,\, ad(E_H))$ of the group of
vertical infinitesimal automorphisms of $E_H$. This gives rise to the following digram of homomorphisms on $E_H$:
\begin{equation}\label{eq:infaut1}
\xymatrix{
& T_{E_H} \ar[r]^A \ar[d]_{L_\psi} & E_H\times\fg \ar[d]^{\psi} \\& T_{E_H}
\ar[r]^A & E_H\times \fg
}
\end{equation}
Diagram \eqref{eq:infaut1} is interpreted as follows. As mentioned before, holomorphic
sections of $ad(E_H)$ are $H$-invariant vertical sections of the holomorphic tangent bundle
$T_{E_H}$. We apply $A$ to a holomorphic section $v$ of $T_{E_H}$; then $\psi$
is applied to $A(v)$, which is simply the Lie bracket on $\mathfrak g$ because
$\psi$ is a function on $E_H$ with values in
$\mathfrak h$. On the other hand, $\psi$ acts on holomorphic sections of $ad(E_H)$ by
applying the Lie derivative $L_\psi$, which is the infinitesimal version of the adjoint
action.

The infinitesimal automorphism $\psi$ is compatible with the holomorphic \cage, if
\eqref{eq:infaut} commutes for all sections $v$ of $ad(E_H)$, i.e.\
$$
A(L_\psi(v))\,=\,\psi(A(v))\,=\, L_\psi(A)(v) + A(L_\psi(v))
$$
for all $v$. Therefore, \eqref{eq:infaut} commutes if and only if $L_\psi(A)\,=\,0$.
\end{proof}

\subsection{Infinitesimal deformations of holomorphic Cartan geometries}

Let $\C[\epsilon]\,=\, \C[t]/(t^2)$, so that $\C[\epsilon]\,=\, \C\oplus \epsilon\cdot \C$ holds with $\epsilon^2 \,=\,0$. The space $D\,=\,(\{0\},\, \cO_D)$ is also called double point with $\cO_D \,=\,\C[\epsilon]$. The tangent space $T_{S,s_0}$ of an arbitrary complex space $S$ at a point $s_0\,\in\, S$ can be identified with the space of all holomorphic mappings $D \,\longrightarrow\, S$ such that the underlying point $0$ is mapped to $s_0$. 

An {\it infinitesimal deformation} is an isomorphism class of deformations over $D$ considered as a complex 
space, or equivalently a deformation over the induced space germ. In particular, an {\it infinitesimal 
deformation} of a Cartan geometry is a deformation over the germ of complex space represented by $D$ in the 
sense of Definition \ref{de:holfam} (ii).

Let $E_H$ be a principal $H$-bundle over $X$, and let $\cE_H$ be an infinitesimal deformation of $E_H$ given 
by an $H$-principal bundle over $X\times D$, whose restriction to $X$ (the fiber over the canonical base 
point) is equipped with an isomorphism to $E_H$ over $X$. We describe such an object in more detail in the 
following:

\begin{proposition}\label{pr:infdef}
An infinitesimal deformation $\cE_H$ of $E_H$ can be described in the following alternative ways.
\begin{itemize}
\item[(i)] The restriction map $\rho\,:\, \cE_H \,\longrightarrow\, E_H$ of $H$-principal bundles defines an affine $\fh$-bundle 
$\mathit{aff}(\cE_H)$ over $E_H$ 
with underlying vector bundle $ad(E_H)$, whose affine structure is determined by a cocycle from $H^1(X,\, ad(E_H))$ up to isomorphism over $E_H$.

\item[(ii)] Let $H_D\,=\, H\oplus \epsilon\, \fh$ be the Lie group defined by $\epsilon^2=0$ and the adjoint action. Then $\cE_H$ is an 
$H_D$-principal bundle on $X$ with an underlying $H$-bundle $E_H$.
\end{itemize}
The adjoint bundle $ad(\cE_H)$ taken from the $H_D$-bundle $\cE_H$ on $X$ possesses the Lie algebra $\fh_D\,=\,\fh \oplus\epsilon\,\fh$ (determined 
by $\epsilon^2=0$) as fiber.
\end{proposition}
\begin{proof}
We need the proof for the fact that isomorphism classes of infinitesimal deformations of $E_H$ correspond to elements of $H^1(X, ad(E_H))$ \cite{Don} in the following version.

Let $\frak U=\{ U_i\}$ be an open covering of $X$ by contractible Stein subsets such that $E_H$ is defined by a cocycle $\{g_{ij}(x)\}$ of $H$-valued, holomorphic mappings on $U_{ij}=U_i\cap U_j$. Now an $H$-principal bundle on $X\times D$ with respect to the covering $U_i\times D$ is determined by a cocycle $\gamma_{ij}$ such that for $x\in U_{ij}$
$$
\wt \gamma_{ij}(x)\,:\, \cO_{H,g_{ij}(x)}\,\longrightarrow\, \cO_{U_{ij},x}\oplus\, \epsilon \, \cO_{U_{ij},x}
$$
with $\wt \gamma_{ij}(x) \,=\, \wt g_{ij}(x) + \epsilon\,\tau_{ij}(x)$. The condition $\epsilon^2\,=\,0$ implies
\begin{equation}\label{eq:Dad}
\wt \gamma_{ij}(x)(\varphi \cdot \psi) = (\varphi \cdot \psi)\circ g_{ij} + \epsilon\left(\tau_{ij}(\varphi)\cdot (\psi\circ g_{ij}) + \tau_{ij}(\psi)\cdot (\varphi\circ g_{ij})\right)
\end{equation}
for all functions $\varphi$ and $\psi$. Hence
$$
\tau_{ik}= L_{g_{ij}*}(\tau_{jk}) + R_{g_{jk}*}(\tau_{ij}),
$$
implies that $\tau_{ij}$ is a derivation, which can be readily written in terms of the Lie-algebra $\fh$ of right-invariant vector fields
$$
\tau_{ik}= ad(g_{ij})(\tau_{jk}) + \tau_{ij}.
$$
The cocycle condition for the transition functions of the $H$-principle bundle $\cE_H$ yields that $\{\tau_{ij}\}$ is an $ad(E_H)$-valued cocycle. Now we can set
$$
\gamma_{ij}\,=\, g_{ij} + \epsilon\, \tau_{ij}\,:\, U_{ij}\times D\,\longrightarrow\, H\, ,
$$
and interpret $\gamma_{ij}$ as a truncated power series in powers of $\epsilon$ with coefficients that are holomorphic on $U_{ij}$. Here $g_{ij}$ has values in $H$ and $\tau_{ij}$ is $\fh$-valued. It can be reinterpreted as an $H_D$-valued cocycle, which proves (ii).

On the other hand, fixing $E_H$ we apply the transition function to a section $\eta_i$ with values in $\fh$. Then the transition equation for an $\fh$-valued holomorphic section $\eta_i$ on $U_i$ is
$$
\eta_j \,\longmapsto\, ad(g_{ij})(\eta_j) + \tau_{ij}\, .
$$
This defines an affine $\fh$-bundle with underlying vector bundle $ad(E_H)$.
\end{proof}

\begin{proposition}\label{pr:affine} Let $(E_H,\, A)$ be a holomorphic \cage\ of type $(G,H)$ on $X$. Let 
${\cE_H}$ be an isomorphism class of infinitesimal deformations of $E_H$. Then the obstructions to extend the 
holomorphic \cage\ $A$ to $\cE_H$ are in $$ H^1(E_H,\, Hom(T_{E_H},\, E_H\times \fg))^H $$ or equivalently in 
$$ H^1(X,\, Hom(At(E_H),\, ad(E_G))\, . $$
\end{proposition}

\begin{proof} We assume that an infinitesimal 
deformation $\cE_H$ of $E_H$ is given like in the Proposition~\ref{pr:infdef}. We use again the notation in 
the proof of Proposition~\ref{pr:infdef} and we consider the affine $\fh$-bundle \begin{equation*} \xymatrix{ 
\cE_H \ar[r]^\rho_\fh & E_H .
}
\end{equation*}
Now $H\subset H_D$ is a subgroup so that over $U_i$ there is a section $\sigma_i$ of $\rho$. In particular there is a $H$-equivariant section of the 
above affine bundle over $\pi^{-1}(U_i)$, where $\pi\,:\,\cE_H \,\longrightarrow\, X$ is the projection on $X$.

Furthermore we get an induced infinitesimal deformation of the tangent bundle $T_{E_H}$.

Now we assume that $(E_H,\, A)$ defines a holomorphic Cartan geometry. Altogether this gives rise to exact sequences on $E_H$
\begin{equation}\label{eq:infdefEH}
\xymatrix{0 \ar[r] &\epsilon\cdot \C \ar[r] \ar@{^{(}->}[d] &
 \C \oplus \epsilon\cdot \C \ar[r] \ar@{^{(}->}[d] & \C \ar[r] \ar@{^{(}->}[d] & 0\\ 0 \ar[r]& \epsilon\cdot T_{E_H} \ar[r]^{T\iota} & T_{\cE_H/D} \ar[r]^{T\rho}\ar@{-->}[d]^{\cA} & T_{E_H}\ar[r]\ar[d]^A_\sim & 0\\
&& {\cE_H}\times \fg \ar[r]_{\rho\times id_\fg}& E_H\times \fg& \: ,
}
\end{equation}
where the existence of an isomorphism $\cA$ has still to be discussed.

Note that the existence of a compatible morphism $\cA$ is equivalent to the existence of a holomorphic Cartan geometry over $D$ (there is only one fiber), which in turn defines an infinitesimal deformation of the given object.

Now we also use the induced splitting $T\sigma_i$ of $T\rho$ over $U_i$. Once the above splittings are fixed there exists an extension 
$\cA_i\,:\,T_{\cE_H/D}|\pi^{-1}(U_i) \,\isom\, E_H\times \fg |\pi^{-1}(U_i)$ of $A|\pi^{-1}(U_i)$. By $\cO_D$-linearity any such map is unique when 
restricted to the kernel of $T\rho|\pi^{-1}(U_i)$. We consider $\cA_i-\cA_j$ over $\pi^{-1}(U_{ij})$. We apply Proposition~\ref{pr:infdef}(i) and see 
that $\cA_i-\cA_j$ has values in $ad(E_H)\times \fg$. Hence it is induced by an $H$-invariant morphism $\cA_{ij}\,:\, T_{E_H}|\pi^{-1}(U_{ij}) 
\,\longrightarrow\,ad(E_H)\times \fg$ as $$ \cA_i\,=\, \cA_{ij}\circ T\rho + \cA_j
$$
over $U_{ij}$, where the addition is taken in the sense of the affine bundle structure $\cE_H\,\longrightarrow\, E_H$.

Suppose that $\cA_{ij}$ is a coboundary of the form $\cB_j-\cB_i|U_{ij}$, where $\cB_i\,:\,T_{E_H}|U_i \,\longrightarrow\, E_H|U_i\times\fg$ are 
linear and 
$H$-invariant. Then, using the affine structure of $\rho: \cE_H \stackrel{\fh}{\longrightarrow} E_H$ the morphisms $\cA_i$ can be changed into $\cA_i + \cB_i \circ T\rho$, so that these fit together, and define the desired (global) map $\cA$.

It can be verified immediately that the cohomology class of $\cA_{ij}$ is uniquely determined by the infinitesimal deformation of the principal $H$-bundle $E_H$ and the holomorphic \cage\ $A$. If we assume that a holomorphic deformation $(\cE_H,\,\cA)$ of the holomorphic Cartan geometry $(E_H,\,A)$ does exist, then $\cA$ is unique up to an $H$-invariant morphism $T_{E_H} \,\longrightarrow\, E_H \times \fg$.
\end{proof}

We denote by $H^1(X,\, ad(E_H))_A$ the space of isomorphism classes of infinitesimal deformations of $E_H$ such that the \cage\ $A$ of the central fiber can be extended, so $H^1(X,\, ad(E_H))_A\, \subset\, H^1(X, \,ad(E_H))$. Furthermore we denote by
$T^1(E_H,A)$ the space of isomorphism classes of infinitesimal deformations of $(E_H,A)$.

Now Proposition~\ref{pr:infaut} and Proposition~\ref{pr:affine} imply the following result.

\begin{theorem}[Infinitesimal deformations]
Let $(E_H,\,A)$ be a holomorphic \cage. Then there is an exact sequence
$$
0 \,\longrightarrow\, H^0(E_H,\, Hom(T_{E_H}, \,E_H\times \fg))^H\,\longrightarrow\, T^1(E_H,\,A)
\,\longrightarrow\, H^1(X,\, ad(E_H))_A\,\longrightarrow\, 0 \, ,
$$
where $H^0(E_H,\, Hom(T_{E_H}, \,E_H\times \fg))^H$ can be identified with $H^0(X,\,Hom(At(E_H),\, ad(E_G)))$.
\end{theorem}

\subsection{Semi-universal deformation of principal bundles}\label{se:supb}

The aim of Section \ref{semi-univ Cartan} is to prove the existence of a semi-universal
deformation of a holomorphic Cartan geometry.
We begin here the construction with a semi-universal deformation of a principal $H$-bundle $E_H$. 

Recall that a deformation over a complex space $S$ with base point
$s_0$ is given by a holomorphic family $\cE_H$ of principal $H$-bundles over $X\times S$, together with an isomorphism $\Xi
\,:\, E_H \,\longrightarrow\, \cE_H|X\times\{s_0\}$. 

{\em Semi-universality} amounts to the following conditions (cf.\ also Definition~\ref{de:holfam} for deformations over germs of complex spaces):
\begin{itemize}
\item[(i)] \textit{(Completeness)} for any deformation $\ul\xi$ of $E_H$ over a space
$(W,w_0)$, given by a holomorphic family $\cE_{H,W} \,\longrightarrow\, X\times W$ together with an
isomorphism of the above type, there is a base change morphism $f\,:\,(W,\,w_0)\,\longrightarrow\, (S,\,s_0)$ such
that (after replacing the base space with open neighborhoods of the base points, if necessary)
the pull back $f^*\ul\xi\,=\,(id_X \times f)^*(\cE_{H})$ and $\cE_{H,W}$ are isomorphic with
isomorphism inducing the identity map over $X\times\{w_0\}$.

\item[(ii)] Let $(W,\,w_0)\,=\,(D,\,0)$ be as in {\rm (i)}. Then any such base change $f$ is {\em uniquely determined} by the given deformation. The set of isomorphism classes of deformations over $(D,0)$ is called ``tangent space of the deformation functor'' or first tangent cohomology associated to the given deformation problem.
\end{itemize}
For the sake of completeness we mention the {\em \ks map}: 

Given a deformation $\ul\xi$ over
$(W,w_0)$ we identify a tangent vector $v$ of $W$ at $w_0$ with a holomorphic map $f\,:\,(D,\,0)\,\longrightarrow\,
(W,\,w_0)$. Then the \ks map $\rho\,:\,T_{W,w_0} \,\longrightarrow\, H^1(X,\,ad(E_H))$ maps $\ul\xi$ to the
isomorphism class of $f^*\ul\xi$. We already computed the tangent cohomologies of order zero
and one.

\section{Semi-universal deformation for holomorphic Cartan geometries}\label{semi-univ Cartan}

In this Section we prove the existence of a semi-universal deformation of a holomorphic Cartan geometry.

Let us first define a natural pull-back functor for holomorphic families and for morphisms between families.

\subsection{A pull-back functor for holomorphic families}\label{su:movb}

Let again $X$ be a compact complex manifold, $S$ a complex space. Let $\mathbf{An}_S$ be the category of complex analytic spaces over $S$. The objects of
$\mathbf{An}_S$ are complex spaces $W$ together endowed with holomorphic maps $W\,\longrightarrow\, S$ and the morphisms are compatible holomorphic mappings.

Denote by $\mathbf{Sets}$ the category of sets.
 
Now we define a {\it pull-back of families} functor $F$ from the category $\mathbf{An}_S$ to the category 
$\mathbf{Sets}$.
 
Choose $E\,,E'$ holomorphic vector bundles on $X\times S$. Consider the map
$$
F \,:\, \mathbf{An}_S\,\longrightarrow\, \mathbf{Sets}
$$
which assigns to each object $(f : W \to S) \,\in\, \mathbf{An}_S$ the set
$$F(f)\,=\,F(f:W\to S) \,=\, Hom((id_X\times f)^*E,\, (id_X\times f)^*E')\, . $$

\begin{proposition}\mbox{}
\begin{itemize}
\item[(i)] The map $F$ is a morphism of categories;
 
\item [(ii)] The functor $F$ is representable.
 
\end{itemize}
\end{proposition}
 
\begin{proof}
(i) Let $\alpha$ be a morphism between the objects $(f_1 : W_1 \to S)$ and $(f_2 : W_2 \to S)$ in the 
category $\mathbf{An}_S$. Then $\alpha$ is given by a holomorphic mapping described in the following diagram:
$$
\xymatrix{
W_1 \ar[r]^\alpha \ar[dr]_{f_1} & W_2\ar[d]^{f_2}\\ &S
}
$$
Let us set $E_{W_i}\,=\, (id_X\times f_i)^*E$ for $i\,=\,1,\,2$ and $E'_{W_i}\,=\, (id_X\times f_i)^*E'$ for $i\,=\,1,\,2$.

Let us denote by $\beta$ the following diagram $$
\quad
\xymatrix{E_{W_2} \ar[r]^\beta \ar[rd] & E'_{W_2} \ar[d]\\
& X \times W_2
}
$$

and notice that the image of $\beta$ through $F(\alpha)$ is given by:
$$
\quad
\xymatrix{E_{W_1} \ar[rr]^{(id_X\times \alpha)^*\beta} \ar[drr] && E'_{W_1} \ar[d]\\
&& X \times W_1&
}
$$

This implies that $F$ is a morphism of categories.

(ii) We shall use the representability of the morphism functor. In the algebraic case it was shown by 
Grothendieck in {\cite[EGAIII 7.7.8 and 7.7.9]{egaIII}}. In the analytic case the corresponding theorem for 
complex spaces is due to Douady \cite[10.1 and 10.2]{dou}. For coherent (locally free) sheaves $\cM$, and 
$\cN$ over a space $\cX\,\longrightarrow\, S$, one considers morphisms of the simple extensions 
$\cX[\cM]\,\longrightarrow\, \cX[\cN]$.

\begin{fact*}The functor $F$ is representable by a complex space $g\,:\,R\,\longrightarrow\, S$. There is a universal object $v$ in $F(g)$
$$
\upsilon:\quad\xymatrix{ E_{R} \ar[r] \ar[dr] & E'_{R}\ar[d]\\
 &X\times R
 }
 $$
 such that any other object over $\wt{R} \,\longrightarrow\, S$ is isomorphic to $\wt g^*\upsilon$, where $\wt{g}\,:\, \wt{R}
\,\longrightarrow\, R $ is a holomorphic map over $S$.
\end{fact*}

\end{proof}

\subsection{Application to holomorphic Cartan geometries}\label{su:appcage}

\begin{theorem}
Let $X$ be a compact complex manifold, $H\,<\,G$ connected complex Lie groups, and $(E_H,\,A)$ a
a holomorphic Cartan geometry on $X$ of type $(G,\,H)$.
Then $(E_H,A)$ possesses a semi-universal deformation.
\end{theorem}

\begin{proof}
We begin with a semi-universal deformation of $E_H$:
$$
\xymatrix{
E_H \ar[r]^\sim \ar[dr]& \cE_H|X\times\{s_0\} \ar@{^{(}->}[r]\ar[d] & \cE_H \ar[d]\\
& X\times \{s_0\}\ar@{^{(}->}[r] & X\times S\; ,
}
$$
and look at the holomorphic vector bundles
$$
\xymatrix{
At(\cE_H) \ar[dr] && ad(\cE_G) \ar[dl]\\
&X\times S&
}\;
$$
to which we apply the representability of the homomorphism functor. We have the following universal homomorphism $\eta$. (It can be checked easily that $At$ and $ad$ commute with base change).
$$
\xymatrix{At(\cE_{H,R}) \ar[rr]^\eta \ar[dr] \ar@/^2pc/[rrr] && ad(\cE_{G,R})\ar@/^2pc/[rrr]\ar[dl]& At(\cE_H)\ar[dr] && ad(\cE_G)\ar[dl]\\
&X\times R \ar[rrr]_{id_X\times h}\ar[d]&&&X\times S\ar[d]&\\
& R \ar[rrr]_h &&& S
}
$$
Now the holomorphic Cartan geometry $A_H\,:\, At(E_H)\,\isom\, ad(E_G)$ amounts to a point $r_0\,\in\,
R$, which is mapped to $s_0$ under $R\,\longrightarrow\, S$ (observing the deformation theoretic isomorphisms of the
given objects and distinguished fibers of the semi-universal families). We take the connected component of $R$ through
$r_0$ and restrict it to an open neighborhood, where the universal morphism is an isomorphism. Going through the
construction, we see that the restricted family yields a semi-universal deformation of the holomorphic \cage\ $(E_H,A)$.
\end{proof}

\section*{Acknowledgements}

We thank the referee for helpful comments to improve the exposition. The first-named author thanks
Universit\'e C\^ote d'Azur and Philipps-Universit\"at Marburg for their hospitality while this
was carried out. He is partially supported by a J. C. Bose Fellowship.

\end{document}